\documentclass[reqno,10pt]{amsart}
\usepackage{amscd,amssymb,verbatim,array}
\usepackage{hyperref}
\usepackage{amsmath}
\usepackage[active]{srcltx} 

\numberwithin{equation}{section} \theoremstyle{plain}

\makeatletter
\@namedef{subjclassname@2020}{%
  \textup{2020} Mathematics Subject Classification}
\makeatother

\newtheorem{theorem}{Theorem}[section]
\newtheorem{lemma}[theorem]{Lemma}

\newtheorem{corollary}[theorem]{Corollary}
\newtheorem{definition}[theorem]{Definition}

\theoremstyle{definition}

\theoremstyle{remark}
\newtheorem{remark}[theorem]{Remark}

\numberwithin{equation}{section}

\newcommand{\Det}{\operatorname{Det}}

\newcommand{\Dir}{\operatorname{D}}

\newcommand{\Dim}{\operatorname{dim}}

\newcommand{\Ker}{\operatorname{ker}}
\newcommand{\Spec}{\operatorname{Spec}}
\newcommand{\Tr}{\operatorname{Tr}}

\newcommand{\Imm}{\operatorname{Im}}

\newcommand{\ddet}{\operatorname{det}}
\newcommand{\Id}{\operatorname{Id}}

\begin{document}

\title[The zeta-determinants and anlaytic torsion of a metric mapping torus]
{The zeta-determinants and anlaytic torsion of a metric mapping torus}

\author{}

\address{}

\email{}

\author{Yoonweon Lee}

\address{Department of Mathematics Education, Inha University, Incheon, 22212, Korea}

\email{yoonweon@inha.ac.kr}

\thanks{This work was supported by INHA UNIVERSITY Research Grant.}

\begin{abstract}
We use the BFK-gluing formula for zeta-determinants to compute the zeta-determinant and analytic torsion of a metric mapping torus induced from an isometry. As applications, we compute the zeta-determinants of the Laplacians defined on a Klein bottle ${\mathbb K}$ and some compact co-K\"ahler manifold ${\mathbb T}_{\varphi}$. We also show that a metric mapping torus and a Riemannian product manifold with a round circle have the same heat trace asymptotic expansions.  We finally compute the analytic torsion of a metric mapping torus for the Witten deformed Laplacian and recover the result of J. Marcsik in \cite{Ma}.
\end{abstract}
\maketitle

\section{Introduction}

\vspace{0.2 cm}

Let $M$ be a compact oriented $(m-1)$-dimensional manifold and $\varphi : M \rightarrow M$ be a diffeomorphism. $\varphi$ may be an orientation preserving or reversing map.
For $a > 0$, we define a mapping torus $M_{\varphi}$ by

\begin{eqnarray}   \label{E:1.1}
M_{\varphi} & = &  M \times [0, ~ a] / (x, 0) \sim (\varphi(x), a).
\end{eqnarray}

\noindent
Equivalently, we define a ${\mathbb Z}$-action on $M \times {\mathbb R}$ by

\begin{eqnarray}   \label{E:1.2}
{\mathbb Z} \times (M \times {\mathbb R}) \rightarrow M \times {\mathbb R}, \qquad m \cdot (x, u) = (\varphi^{m}(x), u + ma ).
\end{eqnarray}

\noindent
This action is properly discontinuous and $M_{\varphi}$ is defined to be the quotient space.
In fact, $M_{\varphi}$ is a fiber bundle over $S^{1}$ and every fiber bundle over $S^{1}$ is obtained in this way.
It is well known that $M_{\varphi}$ is a trivial bundle if and only if $\varphi$ is isotopic to the identity.
To endow a Riemannian metric on $M_{\varphi}$, we are going to consider a more specific case
called a metric mapping torus (\cite{BMO}). Metric mapping tori play important roles in the study of the co-symplectric and co-K\"ahler manifolds (\cite{BO}, \cite{Bl}, \cite{CDM}, \cite{Li}, \cite{MP}), which are odd-dimensional analogues of symplectic and K\"ahler manifolds.

\begin{definition}
Let $(M, g^{M})$ be a compact oriented $(m-1)$-dimensional Riemannian manifold and $\varphi : (M, g^{M}) \rightarrow (M, g^{M})$ be an isometry. A mapping torus $M_{\varphi}$ induced from the Riemannian product $[0, a] \times M$ is called a metric mapping torus.
\end{definition}

A metric mapping torus $M_{\varphi}$ has a natural Riemannian metric and in this paper a mapping torus means a metric mapping torus.
We denote by $\Omega^{q}(M_{\varphi})$ the space of smooth $q$-forms and by $\Delta_{M_{\varphi}}^{q}$ the Hodge Laplacian acting on $\Omega^{q}(M_{\varphi})$.
Equivalently, $\Omega^{q}(M_{\varphi})$ and $\Delta_{M_{\varphi}}^{q}$ are defined in the following way.
We define a map

\begin{eqnarray}    \label{E:1.3}
F : M \times {\mathbb R} \rightarrow M \times {\mathbb R}, \qquad F(x, u) = (\varphi(x), u + a)
\end{eqnarray}

\noindent
Then, $\Omega^{q}(M_{\varphi})$ is identified with

\begin{eqnarray}     \label{E:1.4}
\Omega_{\varphi}^{q}(M \times {\mathbb R}) & := & \{ \omega \in \Omega^{q}(M \times {\mathbb R}) \mid F^{\ast} \omega = \omega \},
\end{eqnarray}

\noindent
and $\Delta_{M_{\varphi}}^{q}$ is identified with $- \frac{d^{2}}{dt^{2}} + \Delta^{q-1}_{M}$ acting on $\Omega_{\varphi}^{q}(M \times {\mathbb R})$.
The zeta function $\zeta_{\Delta_{M_{\varphi}}^{q}}(s)$ associated to $\Delta_{M_{\varphi}}^{q}$ is defined by

\begin{eqnarray}    \label{E:1.5}
\zeta_{\Delta_{M_{\varphi}}^{q}}(s) & = & \sum_{0 \neq \lambda_{j} \in \Spec(\Delta_{M_{\varphi}}^{q})} \lambda_{j}^{-s}
~ = ~ \frac{1}{\Gamma(s)} \int_{0}^{\infty} t^{s-1} \left( \Tr e^{-t \Delta_{M_{\varphi}}^{q} } - \Dim \Ker \Delta_{M_{\varphi}}^{q} \right) ~ dt,
\end{eqnarray}

\noindent
which is holomorphic for $\Re s > \frac{m}{2}$ and has a meromorphic continuation to the whole complex plane ${\mathbb C}$ having a regular value at $s=0$.
If $\Dim \Ker \Delta_{M_{\varphi}}^{q} = 0$, the zeta-determinant $\Det \Delta_{M_{\varphi}}^{q}$ is defined by

\begin{eqnarray}    \label{E:1.6}
\log \Det \Delta_{M_{\varphi}}^{q} & := & - \zeta_{\Delta_{M_{\varphi}}^{q}}^{\prime}(0).
\end{eqnarray}

\noindent
If $\Ker \Delta_{M_{\varphi}}^{q}$ is non-trivial, the modified zeta-determinant $\Det^{\ast} \Delta_{M_{\varphi}}^{q}$ is defined by the same way,
{\it i.e.} $\log \Det^{\ast} \Delta_{M_{\varphi}}^{q} ~ :=$  $- \zeta_{\Delta_{M_{\varphi}}^{q}}^{\prime}(0)$, but we use the upper star to distinguish it from the invertible case.
The analytic torsion $T(M_{\varphi})$ is defined by

\begin{eqnarray}     \label{E:1.7}
\log T(M_{\varphi}) & := & \frac{1}{2} \sum_{q=0}^{m} (-1)^{q+1} q \log \Det^{\ast} \Delta_{M_{\varphi}}^{q}.
\end{eqnarray}

If $\varphi = \Id$, it is easily seen that $M_{\Id} = M \times S^{1}(\frac{a}{2 \pi})$, where $S^{1}(r)$ is the round circle of radius $r$.
In this paper, we are going to use the BFK-gluing formula for zeta-determinants (\cite{BFK}, \cite{Fo}) to describe $\log \Det \Delta_{M_{\varphi}}^{q} - \log \Det \Delta_{M_{\Id}}^{q}$ by a Fredholm determinant of some trace class operator defined on $M$ and determined by $\varphi$. By using the product formula of the zeta-determinants given in Theorem 7.1 of \cite{FG},
we can eventually compute
$\log \Det \Delta_{M_{\varphi}}^{q}$. As applications, we compute the zeta-determinants of the scalar Laplacians defined on a Klein bottle ${\mathbb K}$ and some non-product compact co-K\"ahler manifold ${\mathbb T}_{\varphi}$ given in \cite{CDM}. We also prove that a metric mapping torus and a Riemannian product manifold with a round circle have the same heat trace asymptotic expansions.
We finally apply this result to compute the analytic torsion of $M_{\varphi}$ for the Witten deformed Laplacian and recover the result given by J. Marcsik in \cite{Ma} and D. Burghelea and S. Haller in \cite{BH}. In fact, D. Burghelea and S. Haller proved a more general result
but we here present an elementary computation with a stronger assumption.

\vspace{0.3 cm}

\section{The zeta-determinants of Laplacians on a metric mapping torus}

\vspace{0.2 cm}
In this section we use the BKF-gluing formula to express $\log \Det \Delta_{M_{\varphi}}^{q} - \log \Det \Delta_{M_{\Id}}^{q}$ in terms of a Fredholm determinant of some trace class operator defined on $M$.
We begin with $\Delta_{M_{\varphi}}^{q} + \lambda + \mu$ rather than $\Delta_{M_{\varphi}}^{q}$ for $0 \leq \lambda$, $\mu \in {\mathbb R}$, where $\lambda$ is a fixed real number and $\mu$ is a parameter.
On the $S^{1}$-bundle $p : M_{\varphi} \rightarrow S^{1}$, let $\frac{d}{d \theta}$ and $d \theta$ be the standard unit vector field and one form on $S^{1}$ and let $\frac{\partial}{\partial u}$ and $du$ be the liftings of
$\frac{d}{d \theta}$ and $d \theta$ on $M_{\varphi}$.
Then, for $(x, t) \in M_{\varphi}$ a $q$-form $\omega$ on $M_{\varphi}$ can be expressed by

\begin{eqnarray}    \label{E:2.1}
\omega_{(x, t)} & = & \alpha_{(x, t)} + du \wedge \beta_{(x, t)},
\end{eqnarray}

\noindent
where $\alpha_{(x, t)} \in \wedge^{q}T_{(x, t)} M_{\varphi}$ and $\beta_{(x, t)} \in \wedge^{q-1}T_{(x, t)} M_{\varphi}$ with $\iota_{\frac{\partial}{\partial u}} \alpha_{(x, t)} = \iota_{\frac{\partial}{\partial u}} \beta_{(x, t)} = 0$. For simplicity, we write $x:= (x, 0)$ for $x \in M$.
We note that $M \times \{ 0 \}$ is a submanifold of $M_{\varphi}$ and that

\begin{eqnarray}    \label{E:2.2}
M_{\varphi} - M \times \{ 0 \} & = & M_{\varphi} - M \times \{ a \} ~ = ~ M \times (0, ~ a),
\end{eqnarray}

\noindent
and $M \times \{ 0 \}$ and $M \times \{ a \}$ are identified by $(x, 0) \sim (\varphi(x), a)$, where $\varphi : M \rightarrow M$ is an isometry.
This shows that for a smooth vector field $X$ on $M$ and $x \in M$, $X(x)$ is a vector in $T_{(x, 0)}M_{\varphi}$
and can be identified with $\varphi_{\ast} X(x) \in T_{(\varphi(x), a)}M_{\varphi}$.
For $\alpha \in \Omega^{q}(M)$ and $x \in M$, $\alpha_{x}$ is a $q$-form at $(x, 0) \in M_{\varphi}$, {\it i.e.}  $\alpha_{x} \in \wedge^{q} T^{\ast}_{(x, 0)}M_{\varphi}$.
For smooth vector fields $X_{1}, \cdots, X_{q}$ on $M$ and $x \in M$, the vectors $X_{1}(x), \cdots, X_{q}(x)$ in $T_{(x, 0)}M_{\varphi}$ are
identified with  $\varphi_{\ast} X_{1}(x), \cdots, \varphi_{\ast} X_{q}(x)$ in $T_{(\varphi(x), a)}M_{\varphi}$. Since

\begin{eqnarray}     \label{E:2.3}
\bigg((\varphi^{-1})^{\ast} \alpha_{x} \bigg) \left( \varphi_{\ast} X_{1}(x), \cdots, \varphi_{\ast} X_{q}(x) \right) & = & \alpha_{x} \left( X_{1}(x), \cdots, X_{q}(x) \right),
\end{eqnarray}

\noindent
it follows that $\alpha_{x} \in \wedge^{q} T^{\ast}_{(x, 0)}M_{\varphi}$ is identified with $(\varphi^{-1})^{\ast} \alpha_{x} \in \wedge^{q} T^{\ast}_{(\varphi(x), a)}M_{\varphi}$.
Hence, for $\alpha \in \Omega^{q}(M)$ and $\beta \in \Omega^{q-1}(M)$,
$~ \omega_{x} := \alpha_{x} + du \wedge \beta_{x} \in  \wedge^{q} T^{\ast}_{(x, 0)} M_{\varphi} ~$ is identified with

\begin{eqnarray}    \label{E:2.4}
(\varphi^{-1})^{\ast} \omega_{x} & = & (\varphi^{-1})^{\ast} \alpha_{x} ~ + ~ du \wedge (\varphi^{-1})^{\ast} \beta_{x} ~ \in ~
\wedge^{q} T^{\ast}_{(\varphi(x), a)} M_{\varphi}.
\end{eqnarray}

\vspace{0.2 cm}
For $u \in {\mathbb R}$, we identify $\Omega^{q}(M_{\varphi})|_{M \times \{ u \}}$ with $\Omega^{q}(M) \oplus \Omega^{q-1}(M)$.
We define ${\widetilde \Delta}_{M}^{q}$, $\left( {\widetilde \varphi^{\ast}} \right)_{q}$ and $\left(({\widetilde \varphi}^{-1})^{\ast}\right)_{q}$ as follows.

\begin{eqnarray}    \label{E:2.5}
& & {\widetilde \Delta}_{M}^{q}, ~ ({\widetilde \varphi}^{i})^{\ast}_{q} ~ : ~ \Omega^{q}(M) \oplus \Omega^{q-1}(M) ~ \rightarrow ~ \Omega^{q}(M) \oplus \Omega^{q-1}(M), \\
& & {\widetilde \Delta}_{M}^{q} = \left( \begin{array}{clcr} \Delta_{M}^{q} & 0 \\ 0 & \Delta_{M}^{q-1} \end{array} \right), \quad
({\widetilde \varphi}^{i})^{\ast}_{q} = \left( \begin{array}{clcr} (\varphi^{i})^{\ast}_{q} & 0 \\ 0 & (\varphi^{i})^{\ast}_{q-1} \end{array} \right), \quad i = \pm 1,       \nonumber
\end{eqnarray}

\noindent
where $(\varphi^{i})^{\ast}_{q} : \Omega^{q}(M) \rightarrow \Omega^{q}(M)$ is the pull-back of $q$-forms with respect to $\varphi$ or $\varphi^{-1}$.
For $(\alpha, \beta) \in \Omega^{q}(M) \oplus \Omega^{q-1}(M)$, we embed $(\alpha, \beta)$ to $\wedge^{q} T^{\ast}_{(x, 0)}M_{\varphi}$ and $\wedge^{q} T^{\ast}_{(\varphi(x), a)}M_{\varphi}$ by

\begin{eqnarray}    \label{E:2.6}
\omega_{x} := \alpha_{x} + du \wedge \beta_{x} \in \wedge^{q} T^{\ast}_{(x, 0)}M_{\varphi}, \quad
(\varphi^{-1})^{\ast} \omega_{x} = (\varphi^{-1})^{\ast} \alpha_{x} + du \wedge (\varphi^{-1})^{\ast} \beta_{x} \in \wedge^{q} T^{\ast}_{(\varphi(x), a)}M_{\varphi}.
\end{eqnarray}

\noindent
Let $\widetilde{\omega} \in \Omega^{q}(M \times [0, a])$ be an arbitrary extension of $\omega_{x}$ and $(\varphi^{-1})^{\ast} \omega_{x}$ to $M \times [0, a]$.
We define $\psi(x, u) \in \Omega^{q}(M \times [0, a])$ by

\begin{eqnarray*}
\psi(x, u) & = & \widetilde{\omega} ~ - ~ \left( - \frac{d^{2}}{du^{2}} + \left({\widetilde \Delta^{q}_{M}} + \lambda + \mu \right) \right)^{-1}_{\Dir}
\left( - \frac{d^{2}}{du^{2}} + \left({\widetilde \Delta^{q}_{M}} + \lambda + \mu \right) \right) \widetilde{\omega},
\end{eqnarray*}

\noindent
where "$\Dir$" stands for the Dirichlet boundary condition. Then $\psi(x, u)$ satisfies the following properties.

\begin{eqnarray}    \label{E:2.7}
\left( - \frac{d^{2}}{du^{2}} + \left({\widetilde \Delta^{q}_{M}} + \lambda + \mu \right) \right) \psi(u, x) = 0, \qquad \psi(x, 0) = \omega_{x}, \quad
\psi(\varphi(x), a) = (\varphi^{-1})^{\ast} \omega_{x}.
\end{eqnarray}

\begin{definition}
For $0 \leq \lambda$, $\mu \in {\mathbb R}$, we define the Dirichlet-to-Neumann operator

\begin{eqnarray*}
R^{q}_{\varphi}(\lambda + \mu) : \Omega^{q}(M) \oplus \Omega^{q-1}(M) ~ \rightarrow ~ \Omega^{q}(M) \oplus \Omega^{q-1}(M)
\end{eqnarray*}

\noindent
by

\begin{eqnarray*}
R^{q}_{\varphi}(\lambda + \mu) (\alpha, \beta) & = & \varphi^{\ast} \left( \frac{\partial}{\partial u}\psi(\varphi(x), a) \right) - \frac{\partial}{\partial u}\psi(x, 0) ~ \in ~ \Omega^{q}(M) \oplus \Omega^{q-1}(M),
\end{eqnarray*}

\noindent
where we identify $\frac{\partial}{\partial u}\psi(\varphi(x), a) \in \wedge^{q} T^{\ast}_{(\varphi(x), a)} M_{\varphi}$ with
$\varphi^{\ast} \left( \frac{\partial}{\partial u}\psi(\varphi(x), a) \right) \in \wedge^{q} T^{\ast}_{(x, 0)} M_{\varphi}$ and
identify  $A + du \wedge B$ with $(A, B) \in \Omega^{q}(M) \oplus \Omega^{q-1}(M)$.
\end{definition}

\vspace{0.2 cm}

It is well known that $R^{q}_{\varphi}(\lambda + \mu)$ is a non-negative elliptic pseudodifferential operator of order $1$ with the principal symbol $\sigma_{L}(R^{q}_{\varphi}(\lambda + \mu))(x, \xi) = 2 \sqrt{\parallel \xi \parallel^{2} + \lambda + \mu}$.
We now compute $R^{q}_{\varphi}(\lambda + \mu)$ precisely.
Let $\alpha \in \Omega^{q}(M)$ or $\alpha \in \Omega^{q-1}(M)$ with $\Delta_{M} \alpha = \nu^{2} \alpha$.
We suppose that $\alpha \in \Omega^{q}(M)$ and $\omega_{x} = \alpha_{x}$. Then, the solution to (\ref{E:2.7}) is given by

\begin{eqnarray}    \label{E:2.8}
\psi(x, u)  =
\frac{ e^{- \sqrt{\nu^{2} + \lambda + \mu}(u-a)} - e^{\sqrt{\nu^{2} + \lambda + \mu}(u-a)}}
{e^{a\sqrt{\nu^{2} + \lambda + \mu}} - e^{-a\sqrt{\nu^{2} + \lambda + \mu}}}  \alpha_{x}
  +  \frac{e^{u \sqrt{\nu^{2} + \lambda + \mu}} - e^{- u \sqrt{\nu^{2} + \lambda + \mu}}}
{e^{a\sqrt{\nu^{2} + \lambda + \mu}} - e^{-a\sqrt{\nu^{2} + \lambda + \mu}}}  \left((\varphi^{-1})^{\ast} \alpha \right)_{x}.
\end{eqnarray}

\noindent
Hence,

\begin{eqnarray}    \label{E:2.9}
\frac{\partial}{\partial u}\psi(x, u) ~ = ~
\end{eqnarray}
\begin{eqnarray*}
\sqrt{\nu^{2} + \lambda + \mu} ~ \left\{ ~ - ~
\frac{ e^{- \sqrt{\nu^{2} + \lambda + \mu}(u-a)} +  e^{\sqrt{\nu^{2} + \lambda + \mu}(u-a)}}
{e^{a\sqrt{\nu^{2} + \lambda + \mu}} - e^{-a\sqrt{\nu^{2} + \lambda + \mu}}} ~ \alpha_{x}  ~ + ~
\frac{e^{u \sqrt{\nu^{2} + \lambda + \mu}} + e^{- u \sqrt{\nu^{2} + \lambda + \mu}}}
{e^{a\sqrt{\nu^{2} + \lambda + \mu}} - e^{-a\sqrt{\nu^{2} + \lambda + \mu}}} ~ \left( (\varphi^{-1})^{\ast} \alpha \right)_{x} \right\},
\end{eqnarray*}

\noindent
which yields

\begin{eqnarray*}
& & \frac{\partial}{\partial u}\psi(x, 0) ~ = ~
\frac{\sqrt{\nu^{2} + \lambda + \mu}}{e^{a\sqrt{\nu^{2} + \lambda + \mu}} - e^{-a\sqrt{\nu^{2} + \lambda + \mu}}} ~ \left( - \left(e^{a\sqrt{\nu^{2} + \lambda + \mu}} + e^{-a\sqrt{\nu^{2} + \lambda + \mu}} \right) \alpha_{x} ~ + ~ 2 \left( (\varphi^{-1})^{\ast} \alpha \right)_{x} \right) , \\
& & \frac{\partial}{\partial u}\psi(\varphi(x), a)  =
\frac{\sqrt{\nu^{2} + \lambda + \mu}}{e^{a\sqrt{\nu^{2} + \lambda + \mu}} - e^{-a\sqrt{\nu^{2} + \lambda + \mu}}} \left( - 2 \alpha_{\varphi(x)}  +
 \left( e^{a\sqrt{\nu^{2} + \lambda + \mu}} + e^{-a\sqrt{\nu^{2} + \lambda + \mu}} \right)
 \left( (\varphi^{-1})^{\ast} \alpha \right)_{\varphi(x)} \right)  \\
&  & \hspace{1.5 cm} = ~ (\varphi^{-1})^{\ast} \left\{  \frac{\sqrt{\nu^{2} + \lambda + \mu}}{e^{a\sqrt{\nu^{2} + \lambda + \mu}} - e^{-a\sqrt{\nu^{2} + \lambda + \mu}}} \left( - 2 \left( \varphi^{\ast} \alpha \right)_{x}  +
 \left( e^{a\sqrt{\nu^{2} + \lambda + \mu}} + e^{-a\sqrt{\nu^{2} + \lambda + \mu}} \right)
 \alpha_{x} \right)  \right\}.
\end{eqnarray*}

\vspace{0.2 cm}
\noindent
This leads to the following result.

\begin{eqnarray}     \label{E:2.10}
& & R^{q}_{\varphi}(\lambda + \mu) ~ \alpha (x) ~ = ~ \varphi^{\ast} \left( \frac{\partial}{\partial u}\psi(\varphi(x), a) \right) - \frac{\partial}{\partial u}\psi(x, 0) \\
& = & \left(~ - ~\frac{2 \sqrt{\nu^{2} + \lambda + \mu}}{e^{a\sqrt{\nu^{2} + \lambda + \mu}} - e^{-a\sqrt{\nu^{2} + \lambda + \mu}}} ~ \left( \varphi^{\ast} \alpha \right)_{x} ~ + ~
\sqrt{\nu^{2} + \lambda + \mu} ~ \frac{e^{a\sqrt{\nu^{2} + \lambda + \mu}} + e^{-a\sqrt{\nu^{2} + \lambda + \mu}}}
{e^{a\sqrt{\nu^{2} + \lambda + \mu}} - e^{-a\sqrt{\nu^{2} + \lambda + \mu}}} ~ \alpha_{x} \right)       \nonumber  \\
& & - ~ \left( ~ - ~ \sqrt{\nu^{2} + \lambda + \mu} ~ \frac{e^{a\sqrt{\nu^{2} + \lambda + \mu}} + e^{-a\sqrt{\nu^{2} + \lambda + \mu}}}{e^{a\sqrt{\nu^{2} + \lambda + \mu}} - e^{-a\sqrt{\nu^{2} + \lambda + \mu}}} ~
\alpha_{x}
~ + ~ \frac{2 \sqrt{\nu^{2} + \lambda + \mu}}{e^{a\sqrt{\nu^{2} + \lambda + \mu}} - e^{-a\sqrt{\nu^{2} + \lambda + \mu}}} ~ \left( (\varphi^{-1})^{\ast} \alpha \right)_{x}   \right)      \nonumber  \\
& = & 2 \sqrt{\nu^{2} + \lambda + \mu} ~ \frac{e^{a\sqrt{\nu^{2} + \lambda + \mu}} - 1}{e^{a\sqrt{\nu^{2} + \lambda + \mu}} + 1} \left\{ \alpha_{x} +
\frac{2 e^{a\sqrt{\nu^{2} + \lambda + \mu}}}{\left( e^{a\sqrt{\nu^{2} + \lambda + \mu}} - 1 \right)^{2}}
\left( \alpha_{x} - \left( \frac{\left( \varphi^{\ast} \alpha \right)_{x} + \left( (\varphi^{-1})^{\ast} \alpha \right)_{x}}{2} \right) \right) \right\}.
  \nonumber
\end{eqnarray}

\vspace{0.2 cm}
\noindent
\begin{remark} \label{Remark:2.2}
If $\lambda = \mu =0$ and $\nu = 0$, then $\psi(x, u)$ and $R^{q}_{\varphi}(0) \alpha (x)$ are given by

\begin{eqnarray*}
\psi(x, u) & = & \frac{1}{a} \bigg( \left(( \varphi^{-1})^{\ast} \alpha\right)_{x} - \alpha_{x} \bigg)u + \alpha_{x}, \qquad
R^{q}_{\varphi}(0) \alpha (x) ~ = ~ \frac{2}{a} \left( \alpha_{x} - \frac{\left( \varphi^{\ast} \alpha \right)_{x} + \left( (\varphi^{-1})^{\ast} \alpha \right)_{x}}{2} \right).
\end{eqnarray*}
\end{remark}

\vspace{0.2 cm}
\noindent
The following lemma is well known (for example, Lemma 3.2 in \cite{KL1}).

\begin{lemma}   \label{Lemma:2.3}
For a positive definite elliptic pseudodifferential operator $A$ of positive order and a non-negative trace class operator $Q$, the following equality holds.
\begin{eqnarray*}
\log \Det A ( I + Q) & = & \log \Det A ~ + ~ \log \ddet_{Fr} (I + Q).
\end{eqnarray*}
\end{lemma}

\noindent
We have the following result.

\begin{lemma} \label{Lemma:2.4}
The Dirichlet-to-Neumann operator $ R^{q}_{\varphi}(\lambda + \mu) : \Omega^{q}(M) \oplus \Omega^{q-1}(M) \rightarrow \Omega^{q}(M) \oplus \Omega^{q-1}(M) $ is described as follows.

\begin{eqnarray*}
R^{q}_{\varphi}(\lambda + \mu)  & = &
\end{eqnarray*}
\begin{eqnarray*}
2 \sqrt{{\widetilde \Delta}_{M}^{q} + \lambda + \mu} ~ \frac{e^{a\sqrt{{\widetilde \Delta}_{M}^{q} + \lambda + \mu}} - \Id}{e^{a\sqrt{{\widetilde \Delta}_{M}^{q} + \lambda + \mu}} + \Id}
\left\{ \Id + \frac{2 e^{a\sqrt{{\widetilde \Delta}_{M}^{q} + \lambda + \mu}}}{\left( e^{a\sqrt{{\widetilde \Delta}_{M}^{q} + \lambda + \mu}} - \Id \right)^{2}}
\left( \Id - \frac{1}{2}  {\widetilde \varphi^{\ast}}_{q} - \frac{1}{2} ({\widetilde \varphi}^{-1})^{\ast}_{q} \right) \right\}.
\end{eqnarray*}

\noindent
In particular, if $\varphi = \Id$, then

\begin{eqnarray*}
R^{q}_{\Id}(\lambda + \mu) & =  & 2 \sqrt{{\widetilde \Delta}_{M}^{q} + \lambda + \mu} ~ \frac{e^{a\sqrt{{\widetilde \Delta}_{M}^{q} + \lambda + \mu}} - \Id}{e^{a\sqrt{{\widetilde \Delta}_{M}^{q} + \lambda + \mu}} + \Id} ~ = ~ 2 \sqrt{{\widetilde \Delta}_{M}^{q} + \lambda + \mu} ~ \left( \Id - \frac{2}{e^{a\sqrt{{\widetilde \Delta}_{M}^{q} + \lambda + \mu}} + \Id} \right),
\end{eqnarray*}

\noindent
which shows that for $\lambda + \mu > 0$,

\begin{eqnarray*}
\log \Det R^{q}_{\varphi}(\lambda + \mu) & = & \log \Det R^{q}_{\Id}(\lambda + \mu) \\
& & + ~ \log \ddet_{Fr} \left\{ \Id +
\frac{2 e^{a\sqrt{{\widetilde \Delta}_{M}^{q} + \lambda + \mu}}}{\left( e^{a\sqrt{{\widetilde \Delta}_{M}^{q} + \lambda + \mu}} - \Id \right)^{2}}
\left( \Id -  \frac{1}{2}  {\widetilde \varphi^{\ast}}_{q} - \frac{1}{2} ({\widetilde \varphi}^{-1})^{\ast}_{q} \right) \right\}.
\end{eqnarray*}
\end{lemma}

\vspace{0.2 cm}

We fix $\lambda$ and take $\mu$ as a parameter. By the BFK-gluing formula (\cite{BFK}, \cite{Ca}) we have the following equality.

\begin{eqnarray}    \label{E:2.11}
\log \Det \left( \Delta^{q}_{M_{\varphi}} + \lambda + \mu \right) & = & \sum_{k=0}^{\big[ \frac{m-1}{2} \big]} c_{k}(\varphi, \lambda) \mu^{k} +
\log \Det \left( \Delta^{q}_{[0, a] \times M, \Dir} + \lambda + \mu \right) + \log \Det R^{q}_{\varphi}(\lambda + \mu)     \nonumber \\
& = & \sum_{k=0}^{\big[\frac{m-1}{2}\big]} c_{k}(\varphi, \lambda) \mu^{k} + \log \Det \left( \Delta^{q}_{[0, a] \times M, \Dir} + \lambda + \mu \right) + \log \Det R^{q}_{\Id}(\lambda + \mu)      \nonumber \\
& + &  \log \ddet_{Fr} \left\{ \Id + \frac{2 e^{a\sqrt{{\widetilde \Delta}_{M}^{q} + \lambda + \mu}}}{\left( e^{a\sqrt{{\widetilde \Delta}_{M}^{q} + \lambda + \mu}} - \Id \right)^{2}}
\left( \Id - \frac{1}{2}  {\widetilde \varphi^{\ast}}_{q} - \frac{1}{2} ({\widetilde \varphi}^{-1})^{\ast}_{q} \right) \right\}.
\end{eqnarray}

\noindent
For fixed $\lambda > 0$ and $\mu \rightarrow \infty$, the each term in the above equality has an asymptotic expansion. It is well known that the zero coefficients in the asymptotic expansions of
$\log \Det \left( \Delta^{q}_{M_{\varphi}} + \lambda + \mu \right)$ and $\log \Det \left( \Delta^{q}_{[0, a] \times M, \Dir} + \lambda + \mu \right)$ for $\mu \rightarrow \infty$ are zero (Lemma 2.1 in \cite{KL2}, eq.(1.6) in \cite{KL3}, or eq.(5.1) in \cite{Vo}).
Since

\begin{eqnarray}      \label{E:2.12}
\log \ddet_{Fr} \left\{ \Id + \frac{2 e^{a\sqrt{{\widetilde \Delta}_{M}^{q} + \lambda + \mu}}}{\left( e^{a\sqrt{{\widetilde \Delta}_{M}^{q} + \lambda + \mu}} - 1 \right)^{2}}
\left( \Id - \frac{1}{2}  {\widetilde \varphi^{\ast}}_{q} - \frac{1}{2} ({\widetilde \varphi}^{-1})^{\ast}_{q} \right) \right\} & = & O(e^{- a \sqrt{\mu}}) ,
\end{eqnarray}

\noindent
it follows that $-c_{0}(\varphi, \lambda)$ is the zero coefficient in the asymptotic expansion of $ \log \Det R^{q}_{\Id}(\lambda + \mu)$ for $\mu \rightarrow \infty$,
which shows that $c_{0}(\varphi, \lambda)$ does not depend on $\varphi$.
We compute $c_{0}(\varphi, \lambda) = c_{0}(\lambda)$ in the following way.
From Lemma \ref{Lemma:2.4} we obtain the following equality.

\begin{eqnarray}    \label{E:2.13}
\log \Det R^{q}_{\Id}(\lambda + \mu) & = &  \log 2 \cdot \zeta_{({\widetilde \Delta}_{M}^{q} + \lambda + \mu)}(0) +
\frac{1}{2} \log \Det \left( {\widetilde \Delta}_{M}^{q} + \lambda + \mu \right)    \\
& & + ~ \log \ddet_{Fr} \left( \Id - \frac{2}{e^{a\sqrt{{\widetilde \Delta}_{M}^{q} + \lambda + \mu}} + \Id} \right).  \nonumber
\end{eqnarray}

\noindent
Since $\log \ddet_{Fr} \left( \Id - \frac{2}{e^{a\sqrt{{\widetilde \Delta}_{M}^{q} + \lambda + \mu}} + \Id} \right) = O(e^{-a \sqrt{\mu}})$ and the zero coefficient in the asymptotic expansion of
$\log \Det \left( {\widetilde \Delta}_{M}^{q} + \lambda + \mu \right)$ for $\mu \rightarrow \infty$ is zero,
the zero coefficient of $\log \Det R^{q}_{\Id}(\lambda + \mu)$ is equal to the zero coefficient of $\log 2 \cdot \zeta_{({\widetilde \Delta}_{M}^{q} + \lambda + \mu)}(0)$.
Denoting for $t \rightarrow 0$

\begin{eqnarray}     \label{E:2.14}
\Tr e^{-t{\widetilde \Delta}_{M}^{q}} & \sim & \sum_{j=0}^{\infty} {\frak a}_{j} t^{- \frac{m-1}{2}+ j},
\end{eqnarray}

\noindent
it follows that

\begin{eqnarray}    \label{E:2.15}
\log 2 \cdot \zeta_{({\widetilde \Delta}_{M}^{q} + \lambda + \mu)}(0)  =
\begin{cases} \log 2 \cdot \sum_{\ell = 0}^{\frac{m-1}{2}} \sum_{k=0}^{\frac{m-1}{2} - \ell}  \frac{(-1)^{k+\ell}}{k!\ell!} {\frak a}_{(\frac{m-1}{2} - (k + \ell))} \lambda^{k} \mu^{\ell}  & \quad \text{for} \quad m \quad \text{odd}   \\
0 & \quad \text{for} \quad m \quad \text{even} ,   \end{cases}
\end{eqnarray}

\noindent
which shows that

\begin{eqnarray}   \label{E:2.16}
c_{0}(\lambda) & = &
\begin{cases} - \log 2 \cdot  \sum_{k=0}^{\frac{m-1}{2}} \frac{(-1)^{k}}{k!} {\frak a}_{(\frac{m-1}{2} - k)} \lambda^{k}  & \quad \text{for} \quad m \quad \text{odd}   \\
0 & \quad \text{for} \quad m \quad \text{even} .   \end{cases}
\end{eqnarray}

\noindent
Taking $\mu = 0$ in (\ref{E:2.11}), we obtain the following result.

\begin{eqnarray}     \label{E:2.17}
\log \Det \left( \Delta^{q}_{M_{\varphi}} + \lambda \right) & = & c_{0}(\lambda) + \log \Det \left( \Delta^{q}_{[0, a] \times M, \Dir} + \lambda \right) + \log \Det R^{q}_{\Id}(\lambda) \\
& & + ~  \log \ddet_{Fr} \left\{ \Id + \frac{2 e^{a\sqrt{{\widetilde \Delta}_{M}^{q} + \lambda}}}{\left( e^{a\sqrt{{\widetilde \Delta}_{M}^{q} + \lambda}} - \Id \right)^{2}}
\left( \Id - \frac{1}{2}  {\widetilde \varphi^{\ast}}_{q} - \frac{1}{2} ({\widetilde \varphi}^{-1})^{\ast}_{q} \right) \right\}.  \nonumber
\end{eqnarray}

\noindent
If $\varphi = \Id$, this formula is reduced to the following equality.

\begin{eqnarray}    \label{E:2.18}
\log \Det \left( \Delta^{q}_{M \times S^{1}(\frac{a}{2\pi})} + \lambda \right) & = & c_{0}(\lambda) + \log \Det \left( \Delta^{q}_{[0, a] \times M, \Dir} + \lambda \right) + \log \Det R^{q}_{\Id}(\lambda).
\end{eqnarray}

\noindent
In fact, this equality can be obtained by using Theorem 7.1 of \cite{FG} and Proposition 5.1 of \cite{MM}.
This leads to the following result.

\begin{theorem}   \label{Theorem:2.5}
Let $\varphi : (M, g^{M}) \rightarrow (M, g^{M})$ be an isometry on a compact oriented Riemannian manifold $(M, g^{M})$ and $M_{\varphi}$ be the metric mapping torus associated to $\varphi$.
For $\lambda > 0$, we have the following equality.
\begin{eqnarray*}
& & \log \Det \left( \Delta^{q}_{M_{\varphi}} + \lambda \right)
~ = ~ c_{0}(\lambda) + \log \Det \left( \Delta^{q}_{[0, a] \times M, \Dir} + \lambda \right)
+ \log \Det R^{q}_{\varphi}(\lambda)  \\
& & \hspace{0.5 cm} = ~ \log \Det \left( \Delta^{q}_{M \times S^{1}(\frac{a}{2\pi})} + \lambda \right) ~ + ~
\log \ddet_{Fr} \left\{ \Id + \frac{2 e^{a\sqrt{{\widetilde \Delta}_{M}^{q} + \lambda}}}{\left( e^{a\sqrt{{\widetilde \Delta}_{M}^{q} + \lambda}} - \Id \right)^{2}}
\left( \Id - \frac{1}{2} {\widetilde \varphi^{\ast}}_{q} - \frac{1}{2} ({\widetilde \varphi}^{-1})^{\ast}_{q} \right) \right\},
\end{eqnarray*}

\noindent
where $c_{0}(\lambda)$ is given in (\ref{E:2.16}).
\end{theorem}

We use Theorem \ref{Theorem:2.5} to compare the heat trace asymptotics of $M_{\varphi}$ with that of $M \times S^{1}(\frac{a}{2\pi})$. For $t \rightarrow 0$, we denote

\begin{eqnarray}   \label{E:2.19}
\Tr e^{- t \Delta^{q}_{M_{\varphi}}} & \sim & \sum_{j=0}^{\infty} {\frak b}_{j} t^{- \frac{m}{2} + j}, \qquad
\Tr e^{- t \Delta^{q}_{M \times S^{1}(\frac{a}{2\pi})}} ~ \sim ~ \sum_{j=0}^{\infty} {\frak c}_{j} t^{- \frac{m}{2} + j}.
\end{eqnarray}

\noindent
Simple computation shows that $\log \Det \left( \Delta^{q}_{M_{\varphi}} + \lambda \right)$
has the following asymptotic expansion for $\lambda \rightarrow \infty$ (Lemma 2.1 in \cite{KL2}, eq.(1.6) in \cite{KL3}) .

\begin{eqnarray}    \label{E:2.20}
\log \Det \left( \Delta^{q}_{M_{\varphi}} + \lambda \right) & \sim & - \sum^{N}_{\stackrel{j=0}{j \neq m}} {\frak b}_{j}
\frac{d}{ds} \left( \frac{\Gamma(s - \frac{m-j}{2})}{\Gamma(s)} \right)\bigg|_{s=0} \lambda^{\frac{m-j}{2}} + {\frak b}_{m} \log \lambda   \\
& & + ~ \sum_{j=0}^{m-1} {\frak b}_{j} \left( \frac{\Gamma(s - \frac{m-j}{2})}{\Gamma(s)} \right)\bigg|_{s=0} \lambda^{\frac{m-j}{2}} \log \lambda + O(\lambda^{- \frac{N + 1-m}{2}}).    \nonumber
\end{eqnarray}

\noindent
Similarly, replacing ${\frak b}_{j}$ with ${\frak c}_{j}$ yields the asymptotic expansion of
$\log \Det \left( \Delta^{q}_{M \times S^{1}(\frac{a}{2\pi})} + \lambda \right)$.
Since $\log \ddet_{Fr} \left\{ \Id + \frac{2 e^{a\sqrt{{\widetilde \Delta}_{M}^{q} + \lambda}}}{\left( e^{a\sqrt{{\widetilde \Delta}_{M}^{q} + \lambda}} - 1 \right)^{2}}
\left( \Id - \frac{1}{2} {\widetilde \varphi^{\ast}}_{q} - \frac{1}{2} ({\widetilde \varphi}^{-1})^{\ast}_{q} \right) \right\}$ is $O(e^{- a \sqrt{\lambda}})$ for $\lambda \rightarrow \infty$, Theorem \ref{Theorem:2.5} shows that
$\log \Det \left( \Delta^{q}_{M_{\varphi}} + \lambda \right)$ and $\log \Det \left( \Delta^{q}_{M \times S^{1}(\frac{a}{2\pi})} + \lambda \right)$ have the same asymptotic expansions,
which yields the following result.

\begin{corollary}   \label{Corollary:2.6}
For $t \rightarrow 0$, $\Tr e^{- t \Delta^{q}_{M_{\varphi}}}$ and $\Tr e^{- t \Delta^{q}_{M \times S^{1}(\frac{a}{2\pi})}}$ have the same asymptotic expansions, {\it i.e.} ${\frak b}_{j} = {\frak c}_{j}$.
\end{corollary}

\vspace{0.2 cm}

We next discuss the behavior of each term in Theorem \ref{Theorem:2.5} for $\lambda \rightarrow 0$.
From the definition of $R_{\varphi}^{q}(\lambda)$, it follows that

\begin{eqnarray}    \label{E:2.21}
\Ker R_{\varphi}^{q}(0) & = & \{ \omega|_{{M} \times \{ 0 \}}  \mid \omega \in \Ker \Delta^{q}_{M_{\varphi}} \}, \qquad \Dim \Ker R_{\varphi}^{q}(0) = \Dim \Ker \Delta^{q}_{M_{\varphi}}.
\end{eqnarray}

\noindent
{\it Remark} \ref{Remark:2.2} shows that

\begin{eqnarray}     \label{E:2.22}
\Ker R_{\varphi}^{q}(0) & = & \{ \alpha \in {\mathcal H}^{q}(M) \mid \varphi^{\ast} \alpha = \alpha \} \oplus \{ \beta \in {\mathcal H}^{q-1}(M) \mid \varphi^{\ast} \beta = \beta \}.
\end{eqnarray}

\noindent
We define ${\mathcal H}^{q}(M)$, ${\mathcal H}^{q}_{\varphi}(M)$, ${\mathcal H}^{q}_{\varphi}(M)^{\perp}$ and $b_{q}$, $\ell^{q}_{\varphi}$ as follows.

\begin{eqnarray}    \label{E:2.23}
& & {\mathcal H}^{q}(M) ~ := ~ \{ \sigma \in \Omega^{q}(M) \mid d \sigma = d^{\ast} \sigma = 0 \}, \qquad b_{q}  := \Dim \Ker {\mathcal H}^{q}(M) ,  \\
& & {\mathcal H}^{q}_{\varphi}(M) ~ := ~ \{ \sigma \in {\mathcal H}^{q}(M) \mid \varphi^{\ast} \sigma =  \sigma \}, \qquad \ell^{q}_{\varphi}  := \Dim \Ker {\mathcal H}^{q}_{\varphi}(M),     \nonumber \\
& & {\mathcal S}^{q}_{\varphi}(M) ~ := ~ {\mathcal H}^{q}(M) \ominus  {\mathcal H}^{q}_{\varphi}(M),    \qquad   b_{q} - \ell^{q}_{\varphi} = \Dim \Ker {\mathcal S}^{q}_{\varphi}(M).      \nonumber
\end{eqnarray}

\noindent
Then it follows that

\begin{eqnarray}    \label{E:2.24}
 \Dim \Ker R_{\varphi}^{q}(0) & = & \Dim \Ker \Delta^{q}_{M_{\varphi}} ~ = ~ \ell^{q}_{\varphi} + \ell^{q-1}_{\varphi}.
\end{eqnarray}

\noindent
We note that

\begin{eqnarray}    \label{E:2.25}
& & \log \Det \left( \Delta^{q}_{M_{\varphi}} + \lambda \right) ~ = ~ \left(\ell^{q}_{\varphi} + \ell^{q-1}_{\varphi} \right) \log \lambda + \log \Det^{\ast} \Delta^{q}_{M_{\varphi}}  + O(\lambda),  \\
& & \log \Det \left( \Delta^{q}_{[0, a] \times M, \Dir} + \lambda \right) ~ = ~ \log \Det \Delta^{q}_{[0, a] \times M, \Dir}  + O(\lambda).
\nonumber
\end{eqnarray}

\noindent
For $\alpha \in \Ker R_{\varphi}^{q}(0)$ and $\lambda \rightarrow 0$, Lemma \ref{Lemma:2.4} shows that

\begin{eqnarray}    \label{E:2.26}
R_{\varphi}^{q}(\lambda) ~ \alpha & = & 2 \sqrt{\lambda} \cdot \frac{e^{a \sqrt{\lambda}} - 1}{e^{a \sqrt{\lambda}} + 1} ~ \alpha
~ = ~ a \lambda ~ ( 1 + O(\sqrt{\lambda})) ~ \alpha,
\end{eqnarray}

\noindent
which yields

\begin{eqnarray}    \label{E:2.27}
\log \Det R_{\varphi}^{q}(\lambda) & = & (\ell^{q}_{\varphi} + \ell^{q-1}_{\varphi} ) \log (a \lambda) + \log \Det^{\ast} R_{\varphi}^{q}(0) + O(\lambda).
\end{eqnarray}

\noindent
Eq. (\ref{E:2.16}) shows that

\begin{eqnarray}     \label{E:2.28}
c_{0}(0) & = & - \log 2 \left( \zeta_{{\widetilde \Delta}_{M}^{q}}(0) + \Dim \Ker {\widetilde \Delta}_{M}^{q} \right) ~ = ~
- \log 2 \left( \zeta_{\Delta_{M}^{q}}(0) + \zeta_{\Delta_{M}^{q-1}}(0) + b_{q} + b_{q-1} \right).
\end{eqnarray}

\noindent
Letting $\lambda \rightarrow 0$ in the first equality of Theorem \ref{Theorem:2.5} with (\ref{E:2.25}) - (\ref{E:2.28}), we have the following result.

\begin{eqnarray}     \label{E:2.29}
\log \Det^{\ast} \Delta^{q}_{M_{\varphi}} & = & (\ell^{q}_{\varphi} + \ell^{q-1}_{\varphi}) \log a ~ - ~ \log 2 \left( \zeta_{\Delta_{M}^{q}}(0) + \zeta_{\Delta_{M}^{q-1}}(0) + b_{q} + b_{q-1} \right) \\
& & + ~ \log \Det \Delta^{q}_{[0, a] \times M, \Dir} + \log \Det^{\ast} R_{\varphi}^{q}(0).     \nonumber
\end{eqnarray}

\noindent
Setting $\varphi = \Id$ and using $\ell^{q}_{\Id} = b_{q}$, we have the following result.

\begin{eqnarray}    \label{E:2.30}
\log \Det^{\ast} \Delta^{q}_{\Id} & = & (b_{q} + b_{q-1}) \log a ~ - ~ \log 2 \left( \zeta_{\Delta_{M}^{q}}(0) + \zeta_{\Delta_{M}^{q-1}}(0) + b_{q} + b_{q-1} \right) \\
& & + ~ \log \Det \Delta^{q}_{[0, a] \times M, \Dir} + \log \Det^{\ast} R_{\Id}^{q}(0),    \nonumber
\end{eqnarray}

\noindent
which leads to

\begin{eqnarray}     \label{E:2.31}
\log \Det^{\ast} \Delta^{q}_{M_{\varphi}}  =  (\ell^{q}_{\varphi} + \ell^{q-1}_{\varphi} - b_{q} - b_{q-1}) \log a  +  \log \Det^{\ast} \Delta^{q}_{\Id}  - \log \Det^{\ast} R_{\Id}^{q}(0)
 + \log \Det^{\ast} R_{\varphi}^{q}(0).
\end{eqnarray}

\noindent
Lemma \ref{Lemma:2.4} shows that

\begin{eqnarray*}
R_{\varphi}^{q}(0)  =  2 \sqrt{{\widetilde \Delta}^{q}_{M}} + \frac{ 4 \sqrt{{\widetilde \Delta}^{q}_{M}}}{e^{ 2 a \sqrt{{\widetilde \Delta}^{q}_{M}}} - \Id}
\left( \Id - \frac{1}{2} e^{a \sqrt{{\widetilde \Delta}^{q}_{M}}} \left( {\widetilde \varphi^{\ast}_{q}} + \widetilde{(\varphi^{-1})}^{\ast}_{q} \right) \right),  \quad
R_{\Id}^{q}(0)  =  2 \sqrt{{\widetilde \Delta}^{q}_{M}} - \frac{ 4 \sqrt{{\widetilde \Delta}^{q}_{M}}}{e^{ a \sqrt{{\widetilde \Delta}^{q}_{M}}} + \Id},
\end{eqnarray*}

\noindent
which leads to the following result.

\begin{theorem}   \label{Theorem:2.7}
Under the same assumptions as in Theorem \ref{Theorem:2.5}, the following equality holds.

\begin{eqnarray*}
\log \Det^{\ast} \Delta^{q}_{M_{\varphi}}
& = & - (b_{q} + b_{q-1} - \ell^{q}_{\varphi} - \ell^{q-1}_{\varphi}) \cdot \log \frac{a^{2}}{2} ~ + ~  \log \Det^{\ast} \Delta^{q}_{M \times S^{1}(\frac{a}{2\pi})} \\
& &  + ~  \log \ddet_{Fr} \left( \Id - \frac{1}{2} \left( {\widetilde \varphi^{\ast}_{q}}  + ~ \widetilde{(\varphi^{-1})}^{\ast}_{q} \right) \right)\big|_{{\widetilde {\mathcal S}_{\varphi}^{q}}(M)} \\
&  & + ~ \log \ddet_{Fr} \left\{ \Id + \frac{2 e^{a \sqrt{{\widetilde \Delta_{M}^{q}}}}}{\left( e^{a \sqrt{{\widetilde \Delta_{M}^{q}}}} - \Id \right)^{2}}
\left( \Id - \frac{{\widetilde \varphi}^{\ast}_{q} + \widetilde{(\varphi^{-1})}^{\ast}_{q}}{2} \right) \right\}\big|_{{\widetilde {\mathcal H}}^{q}(M)^{\perp}},
\end{eqnarray*}

\noindent
where ${\widetilde {\mathcal S}_{\varphi}^{q}}(M) = \left( \begin{array}{clcr} {\mathcal S}_{\varphi}^{q}(M) & 0 \\ 0 & {\mathcal S}_{\varphi}^{q-1}(M) \end{array} \right) ~$ and $~ {\widetilde {\mathcal H}^{q}}(M)^{\perp} = \left( \begin{array}{clcr} {\mathcal H}^{q}(M)^{\perp} & 0 \\ 0 & {\mathcal H}^{q-1}(M)^{\perp} \end{array} \right)$.
\end{theorem}

\vspace{0.2 cm}

In the remaining part of this section, we are going to use Theorem \ref{Theorem:2.7} to compute the zeta-determinants of the scalar Laplacians defined on the Klein bottle ${\mathbb K}$ and some compact co-K\"ahler manifold ${\mathbb T}_{\varphi}$ given in \cite{BO} and \cite{CDM}.
The Klein bottle ${\mathbb K}$ is obtained by the group of motions on ${\mathbb R}^{2}$ generated by
$(x, y) \mapsto (x, y + 2 \pi \rho)$ and $(x, y) \mapsto (x + a , 2 \pi \rho - y)$ for some $a > 0$ and $\rho > 0$. The fundamental domain of ${\mathbb K}$ is $(0, a) \times (0, 2 \pi \rho)$.
We denote by $\Omega^{0}_{{\mathbb K}}({\mathbb R}^{2})$ the set of all smooth functions on ${\mathbb R}^{2}$ satisfying

\begin{eqnarray}    \label{E:2.32}
f(x, y) ~ = ~ f(x, y + 2 \pi \rho ) ~ = ~ f(x + a, 2 \pi \rho - y).
\end{eqnarray}

\noindent
Then, the Laplacian $\Delta_{{\mathbb K}}$ is described by

\begin{eqnarray}    \label{E:2.33}
\Delta_{{\mathbb K}} & = & - \left( \frac{\partial^{2}}{\partial x^{2}} + \frac{\partial^{2}}{\partial y^{2}} \right)
\end{eqnarray}

\noindent
defined on $\Omega^{0}_{{\mathbb K}}({\mathbb R}^{2})$.
Let $M = S^{1}(\rho)$ and $\varphi : S^{1}(\rho) \rightarrow S^{1}(\rho)$ be defined by $\varphi(\rho e^{i \theta}) = \rho e^{- i \theta}$. Then, ${\mathbb K} = M_{\varphi}$.
The eigenfunctions are given by

\begin{eqnarray}    \label{E:2.34}
\cos \frac{2 \pi mx}{a}  \cos \frac{ny}{\rho}, \quad \sin \frac{2 \pi mx}{a}  \cos \frac{ny}{\rho}, \quad
\cos \frac{2 \pi \left( m - \frac{1}{2} \right)x}{a}  \sin \frac{ny}{\rho}, \quad \sin \frac{2 \pi \left( m - \frac{1}{2} \right)x}{a}  \sin \frac{ny}{\rho},
\end{eqnarray}

\noindent
and their corresponding eigenvalues are

\begin{eqnarray}   \label{E:2.35}
& & \bigg\{ 0 \bigg\} ~ \cup ~
\bigg\{ \frac{4 \pi^{2} m^{2}}{a^{2}} + \frac{n^{2}}{\rho^{2}} \mid m, n = 1, 2, \cdots \bigg\}
~ \cup ~ \bigg\{ \frac{4 \pi^{2} m^{2}}{a^{2}} \mid m = 1, 2, \cdots \bigg\}  \\
& & \cup ~ \bigg\{ \frac{m^{2}}{\rho^{2}} \mid m = 1, 2, \cdots \bigg\} ~ \cup ~
\bigg\{ \frac{4 \pi^{2} \left( m - \frac{1}{2} \right)^{2}}{a^{2}} + \frac{n^{2}}{\rho^{2}} \mid m, n = 1, 2, \cdots \bigg\},    \nonumber
\end{eqnarray}

\noindent
where the multiplicity of $\frac{m^{2}}{\rho^{2}}$ is $1$ and the multiplicities of other non-zero eigenvalues are $2$. The zeta function on ${\mathbb K}$ is given by

\begin{eqnarray}    \label{E:2.36}
\zeta_{{\mathbb K}}(s) & = & 2 \sum_{n, m=1}^{\infty} \left( \frac{4 \pi^{2} m^{2}}{a^{2}} + \frac{n^{2}}{\rho^{2}} \right)^{-s} ~ + ~ 2 \sum_{m=1}^{\infty} \left(\frac{4 \pi^{2} m^{2}}{a^{2}} \right)^{-s}
~ + ~ \sum_{m=1}^{\infty} \left( \frac{m^{2}}{\rho^{2}} \right)^{-s} \\
& &  + ~ 2 \sum_{n, m=1}^{\infty} \left( \frac{4 \pi^{2} \left( m - \frac{1}{2} \right)^{2}}{a^{2}} + \frac{n^{2}}{\rho^{2}} \right)^{-s}.    \nonumber
\end{eqnarray}

\noindent
Since $\Ker \Delta_{{\mathbb K}}^{0}$, ${\mathcal H}^{0}(S^{1}(\rho))$ and ${\mathcal H}^{0}_{\varphi}(S^{1}(\rho))$
consist of constant functions, it follows that

\begin{eqnarray}    \label{E:2.37}
b_{0} ~ = ~ \ell^{0}_{\varphi} ~ = ~ 1, \qquad  {\mathcal S}_{\varphi}^{0}(S^{1}(\rho)) = \emptyset.
\end{eqnarray}

The flat torus ${\mathbb T}$ is obtained by the group of motions on ${\mathbb R}^{2}$ generated by $(x, y) \mapsto (x, y + 2 \pi \rho)$ and $(x, y) \mapsto (x + a , y)$. The fundamental domain is $(0, a) \times (0, 2 \pi \rho)$.
The eigenfunctions are given by

\begin{eqnarray}    \label{E:2.38}
\cos \frac{2 \pi mx}{a}  \cos \frac{ny}{\rho}, \quad \sin \frac{2 \pi mx}{a}  \cos \frac{ny}{\rho}, \quad
\cos \frac{2 \pi mx}{a}  \sin \frac{ny}{\rho}, \quad \sin \frac{2 \pi mx}{a}  \sin \frac{ny}{\rho},
\end{eqnarray}

\noindent
and their corresponding eigenvalues are

\begin{eqnarray}    \label{E:2.39}
\bigg\{ 0 \bigg\}  \cup
\bigg\{ \frac{4 \pi^{2} m^{2}}{a^{2}} + \frac{n^{2}}{\rho^{2}} \mid m, n = 1, 2, \cdots \bigg\}
 \cup  \bigg\{ \frac{4 \pi^{2} m^{2}}{a^{2}} \mid m = 1, 2, \cdots \bigg\}   \cup  \bigg\{ \frac{m^{2}}{\rho^{2}} \mid m = 1, 2, \cdots \bigg\},
\end{eqnarray}

\noindent
where the multiplicity of $\frac{4 \pi^{2} m^{2}}{a^{2}} + \frac{n^{2}}{\rho^{2}}$ is $4$ and
the multiplicities of other non-zero eigenvalues are $2$.
The zeta function on ${\mathbb T}$ is given by

\begin{eqnarray}    \label{E:2.40}
\zeta_{{\mathbb T}}(s) & = & 4 \sum_{n, m=1}^{\infty} \left( \frac{4 \pi^{2} m^{2}}{a^{2}} + \frac{n^{2}}{\rho^{2}} \right)^{-s} ~ + ~ 2 \sum_{m=1}^{\infty} \left(\frac{4 \pi^{2} m^{2}}{a^{2}} \right)^{-s}
~ + ~ 2 \sum_{m=1}^{\infty} \left( \frac{m^{2}}{\rho^{2}} \right)^{-s}.
\end{eqnarray}

\noindent
It follows that

\begin{eqnarray}    \label{E:2.41}
& & \zeta_{\Delta_{{\mathbb K}}}(s) - \zeta_{\Delta_{{\mathbb T}}}(s) \\
& = &  2 \sum_{n, m=1}^{\infty} \left( \frac{4 \pi^{2} \left( m - \frac{1}{2} \right)^{2}}{a^{2}} + \frac{n^{2}}{\rho^{2}} \right)^{-s}  -  2 \sum_{n, m=1}^{\infty} \left( \frac{4 \pi^{2} m^{2}}{a^{2}} + \frac{n^{2}}{\rho^{2}} \right)^{-s}  -  \sum_{m=1}^{\infty} \left( \frac{m^{2}}{\rho^{2}} \right)^{-s}.
\nonumber
\end{eqnarray}

\vspace{0.2 cm}
\noindent
Theorem \ref{Theorem:2.7} with (\ref{E:2.37}) shows that

\begin{eqnarray}    \label{E:2.42}
\log \Det^{\ast} \Delta_{{\mathbb K}} - \log \Det^{\ast} \Delta_{{\mathbb T}}
 =  \log \ddet_{Fr} \left\{ \Id + \frac{2 e^{a \sqrt{\Delta_{M}^{0}}}}{\left( e^{a \sqrt{\Delta_{M}^{0}}} - \Id \right)^{2}}
\left( \Id - \frac{\varphi^{\ast}_{0} + (\varphi^{-1})^{\ast}_{0}}{2} \right) \right\}\big|_{{\mathcal H}^{0}(S^{1}(\rho))^{\perp}}.
\end{eqnarray}

\noindent
Since $\varphi(\rho e^{i\theta}) = \rho e^{-i\theta}$, we get $\varphi(\rho e^{i\theta}) = \varphi^{-1}(\rho e^{i\theta}) = \rho e^{-i\theta}$
and hence
\begin{eqnarray}      \label{E:4.43}
\left( \Id - \frac{1}{2} \left( \varphi^{\ast}_{0} + (\varphi^{-1})^{\ast}_{0} \right) \right) \cos kx & = & 0, \quad \left( \Id - \frac{1}{2} \left( \varphi^{\ast}_{0} + (\varphi^{-1})^{\ast}_{0} \right) \right) \sin kx ~ = ~ 2 \sin kx.
\end{eqnarray}

\noindent
From $\Delta_{M}^{0} = - \frac{1}{\rho^{2}} \frac{\partial^{2}}{\partial \theta^{2}}$, it follows that

\begin{eqnarray}    \label{E:2.44}
& & \log \ddet_{Fr} \left\{ \Id + \frac{2 e^{a \sqrt{\Delta_{M}^{0}}}}{(e^{a \sqrt{\Delta_{M}^{0}}} - \Id)^{2}}
\left( \Id - \frac{\varphi_{0}^{\ast} + (\varphi_{0}^{-1})^{\ast}}{2} \right) \right\}\big|_{{\mathcal H}^{0}(S^{1}(\rho))^{\perp}} \\
& = & \sum_{k=1}^{\infty} \log \left( 1 + \frac{4 e^{\frac{ak}{\rho}}}{(e^{\frac{ak}{\rho}} - 1)^{2}} \right)
~ = ~  \sum_{k=1}^{\infty} \log \left( \frac{e^{\frac{ak}{\rho}} + 1}{e^{\frac{ak}{\rho}} - 1} \right)^{2} ~ = ~
2 \sum_{k=1}^{\infty} \log \left( 1 + \frac{2}{e^{\frac{ak}{\rho}} - 1} \right).   \nonumber
\end{eqnarray}

\noindent
It is known in Theorem 8.3 of \cite{FG} that

\begin{eqnarray}    \label{E:2.45}
\log \Det^{\ast} \Delta_{{\mathbb T}} & = & 2 \log 2\pi \rho - \frac{2 \pi^{2} \rho}{3a} + 4 \sum_{k=1}^{\infty} \log \left( 1 - e^{- \frac{4 \pi^{2} \rho k}{a}} \right),
\end{eqnarray}

\noindent
which leads to the following result.

\begin{theorem}     \label{Theorem:2.8}
The zeta-determinant of the scalar Laplacian on the Klein bottle ${\mathbb K}$ defined as above is given as follows.
\begin{eqnarray*}
\log \Det^{\ast} \Delta_{{\mathbb K}}  & = &
2 \log 2\pi \rho - \frac{2 \pi^{2} \rho}{3a} + 4 \sum_{k=1}^{\infty} \log \left( 1 - e^{- \frac{4 \pi^{2} \rho k}{a}} \right) + 2 \sum_{k=1}^{\infty} \log \left( 1 + \frac{2}{e^{\frac{ak}{\rho}} - 1} \right).
\end{eqnarray*}
\end{theorem}

\noindent
{\it Remark:} This result was obtained in \cite{Le} by using the different method proved in \cite{KL4}.

\vspace{0.3 cm}
The co-K\"ahler manifolds are odd dimensional analogue of K\"ahler manifolds (\cite{Bl}, \cite{CDM}, \cite{MP}), and
metric mapping tori of K\"ahler manifolds are good examples of compact co-K\"ahler manifolds (\cite{BO}, \cite{Li}).
The first example of a co-K\"ahler manifold ${\mathbb T}^{2}_{\varphi}$ which is not a product with $S^{1}$ is given in \cite{CDM} as follows. Here ${\mathbb T}^{2} := S^{1} \times S^{1}$, where $S^{1}$ is a circle of radius $1$. We define

\begin{eqnarray}     \label{E:2.46}
\varphi : {\mathbb T}^{2} ~ \rightarrow ~ {\mathbb T}^{2}, \qquad \varphi( e^{i \theta}, e^{i \phi}) = (e^{i \phi}, e^{i (\theta + \pi)} ).
\end{eqnarray}

\noindent
Then, $\varphi$ is a Hermitian isometry and ${\mathbb T}^{2}_{\varphi}$ is defined to be the metric mapping torus associated to $\varphi$, {\it i.e.}
${\mathbb T}^{2}_{\varphi} = {\mathbb T}^{2} \times [0, 2\pi]/(x, y, 0) \sim (y, -x, 2 \pi)$.
The Laplacian and its domain are given as follows.

\begin{eqnarray}         \label{E:2.47}
\Delta_{{\mathbb T}^{2}_{\varphi}} = - \frac{d^{2}}{du^{2}} + \Delta_{{\mathbb T}^{2}}, \quad
\Omega_{{\mathbb T}^{2}_{\varphi}} = \big\{ f : {\mathbb T}^{2} \times {\mathbb R} \rightarrow {\mathbb T}^{2} \times {\mathbb R}
\mid f(x, y, u) = f(\varphi(x, y), u + 2 \pi) \big\}.
\end{eqnarray}

\noindent
Since $b_{0} ~ = ~ \ell^{0}_{\varphi} ~ = ~ 1$ and ${\mathcal S}^{0}_{\varphi}({\mathbb T}^{2}) = \emptyset$, Theorem \ref{Theorem:2.7} shows that

\begin{eqnarray}    \label{E:2.48}
& & \log \Det^{\ast} \Delta_{{\mathbb T}^{2}_{\varphi}} ~ - ~ \log \Det^{\ast} \Delta_{{\mathbb T}^{2} \times S^{1}}  \\
& = &  \log \ddet_{Fr} \left\{ \Id + \frac{2 e^{2 \pi \sqrt{\Delta_{{\mathbb T}^{2}}^{0}}}}{\left( e^{2 \pi \sqrt{\Delta_{{\mathbb T}^{2}}^{0}}} - \Id \right)^{2}}
\left( \Id - \frac{\varphi^{\ast}_{0} + (\varphi^{-1})^{\ast}_{0}}{2} \right) \right\}\big|_{{\mathcal H}^{0}({\mathbb T}^{2})^{\perp}}.    \nonumber
\end{eqnarray}

\noindent
We denote $\psi_{m, n}( e^{i \theta}, e^{i \phi}) = e^{i m \theta} e^{i n \phi}$ for $(m, n) \in {\mathbb Z} \times {\mathbb Z}$, which is an eigenfunction of $\Delta_{{\mathbb T}^{2}}$ with eigenvalue $m^{2} + n^{2}$. It follows that

\begin{eqnarray}    \label{E:2.49}
\bigg(\varphi^{\ast}_{0} \psi_{m, n}\bigg)( e^{i \theta}, e^{i \phi}) = (-1)^{n} \psi_{n, m}(e^{i \theta}, e^{i \phi}), \quad
\bigg((\varphi^{-1})^{\ast}_{0} \psi_{m, n}\bigg)( e^{i \theta}, e^{i \phi}) = (-1)^{m} \psi_{n, m}(e^{i \theta}, e^{i \phi}),
\end{eqnarray}

\noindent
which shows that $\Id - \frac{\varphi^{\ast}_{0} + (\varphi^{-1})^{\ast}_{0}}{2}$ has three eigenvalues $\kappa = 0, 1, 2$. The corresponding eigenfunctions are given as follows.

\begin{eqnarray}   \label{E:2.50}
& & \kappa = 0 : \bigg\{ \psi_{2k, 2k} \mid k \in {\mathbb Z} \bigg\}  \cup
\bigg\{ \psi_{2m, 2n} + \psi_{2n, 2m}, ~ \psi_{2m+1, 2n+1} - \psi_{2n+1, 2m+1} \mid n, m \in {\mathbb Z}, n \neq m \bigg\}, \\
& & \kappa = 1 : \bigg\{ \psi_{2m, 2n+1}, ~ \psi_{2n+1, 2m} \mid n, m \in {\mathbb Z} \bigg\},    \nonumber  \\
& & \kappa = 2 : \bigg\{~ \psi_{2k+1, 2k+1} \mid k \in {\mathbb Z} ~ \bigg\} ~ \cup ~  \bigg\{ \psi_{2m, 2n} - \psi_{2n, 2m}, ~ \psi_{2m+1, 2n+1} + \psi_{2n+1, 2m+1} \mid n, m \in {\mathbb Z}, n \neq m \bigg\} .    \nonumber
\end{eqnarray}

\vspace{0.3 cm}

\noindent
Simple computation leads to the following result.

\begin{eqnarray}   \label{E:2.51}
& & \log \Det^{\ast} \Delta_{{\mathbb T}^{2}_{\varphi}} ~ - ~ \log \Det^{\ast} \Delta_{{\mathbb T}^{2} \times S^{1}}  ~ = ~  2 \sum_{(0, 0) \neq (m, n) \in {\mathbb Z} \times {\mathbb Z}} \log \left( 1 + \frac{2}{e^{2 \pi \sqrt{m^{2} + n^{2}}} - 1} \right)  \\
& & \hspace{2.0 cm} + ~ 2 \sum_{(m, n) \in {\mathbb Z} \times {\mathbb Z}} \log \left( 1 - \frac{2 e^{2 \pi \sqrt{4 m^{2} + (2n + 1)^{2}}}}{\left(e^{2 \pi \sqrt{4 m^{2} + (2n + 1)^{2}}} + 1\right)^{2}} \right)
~  - ~ 4 \sum_{m=1}^{\infty} \log \left( 1 + \frac{2}{e^{4 \sqrt{2} \pi m} - 1} \right).   \nonumber
\end{eqnarray}

\vspace{0.3 cm}
\noindent
It is shown in Theorem 7.1 of \cite{FG} that

\begin{eqnarray}    \label{E:2.52}
\log \Det^{\ast} \Delta_{{\mathbb T}^{2} \times S^{1}} & = & 2 \log 2 \pi ~ + ~ 2 \pi \zeta_{{\mathbb T}^{2}}\left( - \frac{1}{2} \right)
~ + ~ 2 \sum_{(0,0) \neq (m, n) \in {\mathbb Z} \times {\mathbb Z}} \log \left( 1 - e^{- 2 \pi \sqrt{m^{2} + n^{2}}} \right).
\end{eqnarray}

\noindent
We finally obtain the following result.

\begin{theorem}     \label{Theorem:2.9}
The zeta-determinant of the scalar Laplacian on ${\mathbb T}^{2}_{\varphi}$ is given as follows.

\begin{eqnarray*}
\log \Det^{\ast} \Delta_{{\mathbb T}^{2}_{\varphi}}  & = & 2 \log 2 \pi ~ + ~ 2 \pi \zeta_{{\mathbb T}^{2}}\left( - \frac{1}{2} \right)
~ + ~ 2 \sum_{(0, 0) \neq (m, n) \in {\mathbb Z} \times {\mathbb Z}} \log \left( 1 + e^{- 2 \pi \sqrt{m^{2} + n^{2}}} \right) \\
& & + ~ 2 \sum_{(m, n) \in {\mathbb Z} \times {\mathbb Z}} \log \left( 1 - \frac{2 e^{2 \pi \sqrt{4 m^{2} + (2n + 1)^{2}}}}{\left(e^{2 \pi \sqrt{4 m^{2} + (2n + 1)^{2}}} + 1\right)^{2}} \right) ~ - ~
4 \sum_{m=1}^{\infty} \log \left( 1 + \frac{2}{e^{4 \sqrt{2} \pi m} - 1} \right).
\end{eqnarray*}
\end{theorem}

\vspace{0.3 cm}

\section{The analytic torsion on a metric mapping torus}

\vspace{0.2 cm}

In this section we use Theorem \ref{Theorem:2.7} to compute the analytic torsion $\log T(M_{\varphi})$ of a metric mapping torus.
Simple computation shows that

\begin{eqnarray}    \label{E:3.1}
& & - \frac{1}{2} \sum_{q=0}^{m} (-1)^{q+1} \cdot q \cdot (b_{q} + b_{q-1} - \ell^{q}_{\varphi} - \ell^{q-1}_{\varphi}) \cdot \log \frac{a^{2}}{2} ~ = ~
- \frac{1}{2} \sum_{q=0}^{m-1} (-1)^{q} \cdot (b_{q} - \ell^{q}_{\varphi}) \cdot \log \frac{a^{2}}{2},  \\
& & \frac{1}{2} \sum_{q=0}^{m} (-1)^{q+1} \cdot q \cdot \log \ddet_{Fr} \left( \Id - \frac{1}{2} \left( {\widetilde \varphi^{\ast}_{q}}  + ~ \widetilde{(\varphi^{-1})}^{\ast}_{q} \right) \right)\big|_{{\widetilde {\mathcal S}_{\varphi}^{q}}(M)}   \nonumber \\
& = & \frac{1}{2} \sum_{q=0}^{m} (-1)^{q+1} \cdot q \cdot \log \ddet_{Fr} \left( \Id - \frac{1}{2} \left( \varphi^{\ast}_{q}  + ~ (\varphi^{-1})^{\ast}_{q} \right) \right)\big|_{{\mathcal S_{\varphi}^{q}}(M)}   \nonumber \\
& & \hspace{1.0 cm} + ~ \frac{1}{2} \sum_{q=0}^{m} (-1)^{q+1} \cdot q \cdot \log \ddet_{Fr} \left( \Id - \frac{1}{2} \left( \varphi^{\ast}_{q-1}  + ~ (\varphi^{-1})^{\ast}_{q-1} \right) \right)\big|_{{\mathcal S}_{\varphi}^{q-1}(M)}   \nonumber \\
& = & \frac{1}{2} \sum_{q=0}^{m-1} (-1)^{q} \cdot \log \ddet_{Fr} \left( \Id - \frac{1}{2} \left( \varphi^{\ast}_{q}  + ~ (\varphi^{-1})^{\ast}_{q} \right) \right)\big|_{{\mathcal S}_{\varphi}^{q}(M)}.   \nonumber
 \nonumber
\end{eqnarray}

\noindent
By the Hodge decomposition, $\Omega^{q}(M)$ is decomposed into

\begin{eqnarray}   \label{E:3.2}
\Omega^{q}(M) & = & \Omega^{q}_{-}(M) \oplus {\mathcal H}^{q}(M) \oplus \Omega^{q}_{+}(M),
\end{eqnarray}

\noindent
where $\Omega^{q}_{-}(M) = \Imm d \cap \Omega^{q}(M)$ and $\Omega^{q}_{+}(M) = \Imm d^{\ast} \cap \Omega^{q}(M)$.
Since $\varphi : M \rightarrow M$ is an isometry, $\varphi^{\ast}_{q}$ and $(\varphi^{-1})^{\ast}_{q}$ preserve $\Omega^{q}_{\pm}(M)$.
We denote by $\Delta_{M}^{q, \pm}$, $\varphi^{\ast, \pm}_{q}$,  $(\varphi^{-1})^{\ast, \pm}_{q}$ the restrictions of $\Delta_{M}^{q}$, $\varphi^{\ast}_{q}$,  $(\varphi^{-1})^{\ast}_{q}$ to $\Omega^{q, \pm}(M)$.
Since $\Delta_{M}^{q}$ and $(\varphi^{i})^{\ast}_{q}$ ($ i = 1, -1$) commute with $d$, it follows that $\Delta_{M}^{q, +}$ and $(\varphi^{i})^{\ast, +}_{q}$ are isospectral with $\Delta_{M}^{q+1, -}$ and $(\varphi^{i})^{\ast, -}_{q+1}$, respectively.
Taking these facts into consideration, we have the following equalities.

\begin{eqnarray}    \label{E:3.3}
& & \frac{1}{2} \sum_{q=0}^{m} (-1)^{q+1} \cdot q \cdot \log \ddet_{Fr} \left\{ \Id + \frac{2 e^{a \sqrt{{\widetilde \Delta_{M}^{q}}}}}{\left( e^{a \sqrt{{\widetilde \Delta_{M}^{q}}}} - \Id \right)^{2}}
\left( \Id - \frac{{\widetilde \varphi}^{\ast}_{q} + \widetilde{(\varphi^{-1})}^{\ast}_{q}}{2} \right) \right\}\big|_{{\widetilde {\mathcal H}}^{q}(M)^{\perp}}   \\
& = & \frac{1}{2} \sum_{q=0}^{m} (-1)^{q+1} \cdot q \cdot \log \ddet_{Fr} \left\{ \Id + \frac{2 e^{a \sqrt{\Delta_{M}^{q}}}}{\left( e^{a \sqrt{\Delta_{M}^{q}}} - \Id \right)^{2}}
\left( \Id - \frac{\varphi^{\ast}_{q} + (\varphi^{-1})^{\ast}_{q}}{2} \right) \right\}\big|_{{\mathcal H}^{q}(M)^{\perp}} \nonumber   \\
&  & + ~ \frac{1}{2} \sum_{q=0}^{m} (-1)^{q+1} \cdot q \cdot \log \ddet_{Fr} \left\{ \Id + \frac{2 e^{a \sqrt{\Delta_{M}^{q-1}}}}{\left( e^{a \sqrt{\Delta_{M}^{q-1}}} - \Id \right)^{2}}
\left( \Id - \frac{\varphi^{\ast}_{q-1} + (\varphi^{-1})^{\ast}_{q-1}}{2} \right) \right\}\big|_{{\mathcal H}^{q-1}(M)^{\perp}} \nonumber  \\
& = & \frac{1}{2} \sum_{q=0}^{m} (-1)^{q} \cdot \log \ddet_{Fr} \left\{ \Id + \frac{2 e^{a \sqrt{\Delta_{M}^{q}}}}{\left( e^{a \sqrt{\Delta_{M}^{q}}} - \Id \right)^{2}}
\left( \Id - \frac{\varphi^{\ast}_{q} + (\varphi^{-1})^{\ast}_{q}}{2} \right) \right\}\big|_{{\mathcal H}^{q}(M)^{\perp}} \nonumber   \\
& = & \frac{1}{2} \sum_{q=0}^{m} (-1)^{q} \cdot \log \ddet_{Fr} \left\{ \Id + \frac{2 e^{a \sqrt{\Delta_{M}^{q, -}}}}{\left( e^{a \sqrt{\Delta_{M}^{q, -}}} - \Id \right)^{2}}
\left( \Id - \frac{\varphi^{\ast}_{q, -} + (\varphi^{-1})^{\ast}_{q, -}}{2} \right) \right\} \nonumber   \\
&  & + ~ \frac{1}{2} \sum_{q=0}^{m} (-1)^{q} \cdot \log \ddet_{Fr} \left\{ \Id + \frac{2 e^{a \sqrt{\Delta_{M}^{q, +}}}}{\left( e^{a \sqrt{\Delta_{M}^{q, +}}} - \Id \right)^{2}}
\left( \Id - \frac{\varphi^{\ast}_{q, +} + (\varphi^{-1})^{\ast}_{q, +}}{2} \right) \right\}  \nonumber \\
& = & 0.  \nonumber
\end{eqnarray}

\vspace{0.2 cm}
\noindent
When $M_{1}$ and $M_{2}$ are compact orienred Riemannian manifolds, it is well known ((7.4) in \cite{Mu}) that

\begin{eqnarray}    \label{E:3.5}
\log T(M_{1} \times M_{2}) & = & \chi(M_{2}) \log T(M_{1}) + \chi(M_{1}) \log T(M_{2}).
\end{eqnarray}

\noindent
When $S^{1}(r)$ is a round circle of radius $r>0$, the Laplacian $\Delta_{S^{1}(r)}$ and its associated zeta function $\zeta_{S^{1}(r)}(s)$ are given by

\begin{eqnarray}    \label{E:3.6}
\Delta_{S^{1}(r)} & = & - \frac{1}{r^{2}} \frac{\partial^{2}}{\partial \theta^{2}}, \qquad
\zeta_{S^{1}(r)}(s) ~ = ~ 2 \sum_{k=1}^{\infty} \left( \frac{k^{2}}{r^{2}} \right)^{-s} ~ = ~ 2 r^{2s} \zeta_{R}(2s),
\end{eqnarray}

\noindent
which shows that

\begin{eqnarray}    \label{E:3.7}
\zeta_{S^{1}(r)}^{\prime}(0) & = & 4 \log r ~ \zeta_{R}(0) + 4 \zeta_{R}^{\prime}(0) ~ = ~ - 2 \log r - 2 \log 2 \pi ~ = ~ -2 \log 2 \pi r.
\end{eqnarray}

\noindent
Hence, it follows that

\begin{eqnarray}     \label{E:3.8}
& & T(S^{1}(r)) ~ = ~ \frac{1}{2} \log \Det^{\ast} \Delta_{S^{1}(r)} ~ = ~  \log 2 \pi r,  \quad
\log T \left(M \times S^{1}(\frac{a}{2\pi}) \right) ~ = ~ \chi(M) \cdot \log a.
\end{eqnarray}

\noindent
Summarizing the above argument, we have the following result.

\begin{theorem}   \label{Theorem:3.1}
Let $\varphi : (M, g^{M}) \rightarrow (M, g^{M})$ be an isometry on a compact oriented Riemannian manifold $(M, g^{M})$ and
$M_{\varphi}$ be the metric mapping torus associated to $\varphi$. Then,

\begin{eqnarray*}
\log T(M_{\varphi}) & = &  \frac{1}{2} \log 2  \cdot \chi(M) + \frac{1}{2} \log \frac{a^{2}}{2} \cdot \sum_{q=0}^{m-1} (-1)^{q} \ell^{q}_{\varphi}  \\
& & ~  + ~  \frac{1}{2} ~ \sum_{q=0}^{m-1} (-1)^{q} \log \ddet_{Fr} \left( \Id - \frac{1}{2} \left( \varphi^{\ast}_{q} + (\varphi^{-1})^{\ast}_{q} \right) \right)\big|_{{\mathcal S}^{q}_{\varphi}(M)}.
\end{eqnarray*}
In particular, if $\Dim M$ is odd and $\varphi$ is orientation preserving, then $\log T(M_{\varphi}) = 0$.
\end{theorem}

\vspace{0.2 cm}
In the remaining part of this section, we are going to compute the analytic torsion for the Witten deformed Laplacian (\cite{Wi}) and recover the result of J. Marcsik given in \cite{Ma}.
Since a mapping torus is a fiber bundle over $S^{1}$, we write the fiber bundle by $p : M_{\varphi} \rightarrow S^{1}(\frac{a}{2 \pi})$.
Let $d \theta$ be the standard one form on $S^{1}(\frac{a}{2 \pi})$ and put $du = p^{\ast} d \theta$.
Setting $d_{q}(t) := d + t du \wedge  : \Omega^{q}(M_{\varphi}) \rightarrow \Omega^{q+1}(M_{\varphi})$, it follows that $d_{q+1}(t) d_{q}(t) = 0$.
We consider a complex  $\left( \Omega^{\ast}(M_{\varphi}), d(t) \right)$ with $\Delta^{q}_{M_{\varphi}}(t) := d^{\ast}_{q}(t) d_{q}(t) + d_{q-1}(t) d^{\ast}_{q-1}(t)$.
Since $d_{q}^{\ast}(t) = d^{\ast} + t \iota_{\frac{\partial}{\partial u}}$, simple computation shows that

\begin{eqnarray}     \label{E:3.9}
\Delta^{q}_{M_{\varphi}}(t) & = & \Delta^{q}_{M_{\varphi}} + t^{2}.
\end{eqnarray}

\noindent
We define the analytic torsion $T(M_{\varphi}, t)$ with respect to $\Delta^{q}_{M_{\varphi}}(t)$ by

\begin{eqnarray}   \label{E:3.10}
\log T \left( M_{\varphi}, t \right) & = & \frac{1}{2} \sum_{q=0}^{m} (-1)^{q+1} \cdot q \cdot \log \Det \left(\Delta_{M_{\varphi}}^{q} + t^{2} \right).
\end{eqnarray}

\noindent
To describe the result of J. Marcsik, we define the Lefschetz number as follows. Let $F : M \rightarrow M$ be a continuous map with the induced homomorphism
$F^{\ast}_{q} : H^{q}(M, {\mathbb R}) \rightarrow H^{q}(M, {\mathbb R})$.  The Lefschetz number $L(F)$ and
Lefschetz zeta function $\zeta_{F}(t)$ are defined by

\begin{eqnarray}     \label{E:3.11}
L(F) & = & \sum_{q=0}^{m-1} (-1)^{q} \Tr F^{\ast}_{q}, \qquad
\log \zeta_{F}(t) = \sum_{k=1}^{\infty} L(F^{k}) \frac{t^{k}}{k}.
\end{eqnarray}

\begin{lemma}  \label{Lemma:3.2}
If $\psi : M \rightarrow M$ is an isometry, then $L(\psi) = L(\psi^{-1})$.
\end{lemma}

\begin{proof}
We note that $\psi^{\ast}_{q} : H^{q}(M, {\mathbb R}) \rightarrow H^{q}(M, {\mathbb R})$ can be identified with $\psi^{\ast}_{q} : {\mathcal H}^{q}(M) \rightarrow {\mathcal H}^{q}(M)$, where
${\mathcal H}^{q}(M)$ is the set of harmonic $q$-forms. Since $\psi$ is an isometry, it follows that for any $v \in {\mathcal H}^{q}(M)$, $\parallel v \parallel = \parallel \psi^{\ast}_{q} v \parallel$.
Hence, for any orthonormal basis for ${\mathcal H}^{q}(M)$, $\psi^{\ast}_{q}$ is expressed by an orthogonal matrix, whose inverse is equal to its transpose. This shows that $\Tr \psi^{\ast}_{q} = \Tr (\psi^{-1})^{\ast}_{q}$, and
the result follows.
\end{proof}

We are going to express $\log T \left( M_{\varphi}, t \right)$ in terms of the Euler characteristic $\chi(M)$ and Lefschetz zeta function $\zeta_{\varphi}(t)$ for $\varphi : M \rightarrow M$. The second equality of Theorem \ref{Theorem:2.5} shows that

\begin{eqnarray}   \label{E:3.12}
\log \Det \left(\Delta_{M_{\varphi}}^{q} + t^{2} \right) & = & \log \Det \left(\Delta_{M \times S^{1}(\frac{a}{2 \pi})}^{q} + t^{2} \right) \\
& & \hspace{0.5 cm} + ~
\log \ddet_{Fr} \left\{ \Id + \frac{2 e^{a\sqrt{{\widetilde \Delta}^{q}_{M} + t^{2}}}}{\left( e^{a\sqrt{{\widetilde \Delta}^{q}_{M} + t^{2}}} - \Id \right)^{2}}
\left( \Id - \frac{{\widetilde \varphi}_{q}^{\ast} + (\widetilde{\varphi^{-1}})_{q}^{\ast}}{2} \right)  \right\},  \nonumber
\end{eqnarray}

\noindent
which leads to

\begin{eqnarray}     \label{E:3.13}
\log T \left( M_{\varphi}, t \right) & = & \log T \left( M \times S^{1}\left(\frac{a}{2 \pi}\right), t \right) \\
& + & \frac{1}{2} \sum_{q=0}^{m} (-1)^{q+1} \cdot q \cdot
\log \ddet_{Fr} \left\{ \Id + \frac{2 e^{a\sqrt{{\widetilde \Delta}^{q}_{M} + t^{2}}}}{\left( e^{a\sqrt{{\widetilde \Delta}^{q}_{M} + t^{2}}} - \Id \right)^{2}}
\left( \Id - \frac{\widetilde{\varphi}_{q}^{\ast} + (\widetilde{\varphi^{-1}})_{q}^{\ast}}{2} \right)  \right\}.    \nonumber
\end{eqnarray}

\noindent
It is a well known fact ((7.4) in \cite{Mu}) that

\begin{eqnarray}      \label{E:3.14}
\log T \left( M \times S^{1}\left(\frac{a}{2 \pi}\right), t \right) & = & \chi(M) \cdot \log T \left( S^{1}\left(\frac{a}{2 \pi}\right), t \right) ~ + ~ \chi \left( S^{1}\left(\frac{a}{2 \pi}\right)\right) \cdot
\log T ( M , t)     \nonumber \\
& = & \frac{1}{2} \chi(M) \cdot \log \Det \left(\Delta_{S^{1}(\frac{a}{2 \pi})} + t^{2} \right).
\end{eqnarray}

\begin{lemma}    \label{Lemma:3.3}
\begin{eqnarray*}
\log \Det \left( \Delta_{S^{1}\left(\frac{a}{2 \pi}\right)} + t^{2} \right) & = &  a t ~ + ~ 2 \log ( 1 - e^{-at}) .
\end{eqnarray*}
\end{lemma}

\begin{proof}

\noindent
Since $\Spec \left( \Delta_{S^{1}(\frac{a}{2 \pi})} \right)  =  \left\{ \left( \frac{2 \pi k}{a} \right)^{2} \mid k \in {\mathbb Z} \right\}$, it follows that

\begin{eqnarray*}
\zeta_{\left(\Delta_{S^{1}(\frac{a}{2 \pi})}^{1} + t^{2}\right)}(s) ~ = ~  2 \sum_{k=1}^{\infty} \left( \left( \frac{2 \pi k}{a} \right)^{2} + t^{2}  \right)^{-s}  ~ + ~ t^{-2s}
~ = ~ \frac{2}{\Gamma(s)} \int_{0}^{\infty} u^{s-1} e^{-u t^{2}} \sum_{k=1}^{\infty} e^{- \frac{4 \pi^{2} k^{2}}{a^{2}} u } ~ du ~ + ~  t^{-2s}.
\end{eqnarray*}

\noindent
The Poisson summation formula shows that

\begin{eqnarray*}
\sum_{k=1}^{\infty} e^{- \frac{4 \pi^{2} k^{2}}{a^{2}}u} & = & - \frac{1}{2} ~ + ~ \frac{a}{4 \sqrt{\pi u}} ~ + ~ \frac{a}{2 \sqrt{\pi u}} ~ \sum_{k=1}^{\infty} e^{- \frac{a^{2}k^{2}}{4 u}} ,
\end{eqnarray*}

\noindent
which yields that

\begin{eqnarray*}
\zeta_{\left(\Delta_{S^{1}(\frac{a}{2 \pi})}^{1} + t^{2}\right)}(s) & = &  \frac{a \Gamma(s - \frac{1}{2})}{2 \sqrt{\pi} \Gamma(s)} t^{-2s + 1} ~ + ~
\frac{a}{\sqrt{\pi}} ~\sum_{k=1}^{\infty} \frac{1}{\Gamma(s)} \int_{0}^{\infty} u^{s-\frac{3}{2}} e^{-(ut^{2} + \frac{a^{2} k^{2}}{4u})} ~ du  .
\end{eqnarray*}

\noindent
Hence,

\begin{eqnarray*}
\zeta^{\prime}_{\left(\Delta_{S^{1}(\frac{a}{2 \pi})}^{1} + t^{2} \right)}(0) & = &  - a t ~ + ~
\frac{a}{\sqrt{\pi}} ~\sum_{k=1}^{\infty} \int_{0}^{\infty} u^{-\frac{3}{2}} e^{-(ut^{2} + \frac{a^{2} k^{2}}{4u})} ~ du  ~ = ~ - a t - 2 \log ( 1 - e^{-at}).
\end{eqnarray*}
\end{proof}

\noindent
Using the similar method as in (\ref{E:3.1}) and (\ref{E:3.3}), we have the following equalities.

\begin{eqnarray}   \label{E:3.15}
& & \frac{1}{2} \sum_{q=0}^{m} (-1)^{q+1} \cdot q \cdot
\log \ddet_{Fr} \left\{ \Id + \frac{2 e^{a\sqrt{{\widetilde \Delta}^{q}_{M} + t^{2}}}}{\left( e^{a\sqrt{{\widetilde \Delta}^{q}_{M} + t^{2}}} - \Id \right)^{2}}
\left( \Id - \frac{\widetilde{\varphi}_{q}^{\ast} + (\widetilde{\varphi^{-1}})_{q}^{\ast}}{2} \right)  \right\}    \\
& = & \frac{1}{2} \sum_{q=0}^{m-1} (-1)^{q}
\log \ddet_{Fr} \left\{ \Id + \frac{2 e^{at}}{\left( e^{at} - 1 \right)^{2}} \left( \Id - \frac{\varphi_{q}^{\ast} + (\varphi_{q}^{-1})^{\ast}}{2} \right)  \right\}\bigg|_{{\mathcal H}^{q}(M)}    \nonumber  \\
& = & - \chi(M) \log \left( 1 - e^{-at} \right) ~ + ~
\frac{1}{2} \sum_{q=0}^{m-1} (-1)^{q} \log \ddet_{Fr} \left( \Id - e^{-at} \varphi_{q}^{\ast} \right) \left( \Id - e^{-at} (\varphi_{q}^{-1})^{\ast} \right)    \nonumber  \\
& = & - \chi(M) \log \left( 1 - e^{-at} \right) ~ + ~
\frac{1}{2} \sum_{q=0}^{m-1} (-1)^{q} \Tr \log \left( \Id - e^{-at} \varphi_{q}^{\ast} \right)  ~ + ~
\frac{1}{2} \sum_{q=0}^{m-1} (-1)^{q} \Tr \log \left( \Id - e^{-at} (\varphi_{q}^{-1})^{\ast} \right)     \nonumber  \\
& = & - \chi(M) \log \left( 1 - e^{-at} \right) ~ - ~
\frac{1}{2} \sum_{q=0}^{m-1} (-1)^{q} \Tr \sum_{k=1}^{\infty} \frac{e^{-akt}}{k} (\varphi_{q}^{\ast})^{k}  ~ - ~
\frac{1}{2} \sum_{q=0}^{m-1} (-1)^{q} \Tr \sum_{k=1}^{\infty} \frac{e^{-akt}}{k} ((\varphi_{q}^{-1})^{\ast})^{k}    \nonumber \\
& = & - \chi(M) \log \left( 1 - e^{-at} \right) ~ - ~ \frac{1}{2} \sum_{k=1}^{\infty} \frac{e^{-akt} L(\varphi^{k})}{k} ~ - ~ \frac{1}{2} \sum_{k=1}^{\infty} \frac{e^{-akt} L((\varphi^{-1})^{k})}{k}     \nonumber \\
& = & - \chi(M) \log \left( 1 - e^{-at} \right) ~ - ~ \sum_{k=1}^{\infty} \frac{L(\varphi^{k}) e^{-akt}}{k},   \nonumber
\end{eqnarray}

\noindent
where we used Lemma \ref{Lemma:3.3} in the last equality.
Finally, we obtain the following result.

\begin{theorem}   \label{Theorem:3.4}
Let $\varphi : (M, g^{M}) \rightarrow (M, g^{M})$ be an isometry on a compact oriented Riemannian manifold $(M, g^{M})$ and
$M_{\varphi}$ be the metric mapping torus associated to $\varphi$. Then,

\begin{eqnarray*}
\log T \left( M_{\varphi}, t \right) & = & \frac{a}{2} ~ \chi(M) ~ t  ~ - ~ \sum_{k=1}^{\infty} \frac{L(\varphi^{k}) e^{-akt}}{k} ~ = ~   \frac{a}{2} ~ \chi(M) ~ t  ~ - ~ \log \zeta_{\varphi}(e^{-at}).
\end{eqnarray*}
In particular, if $\Dim M$ is odd and $\varphi$ is orientation preserving, then $\log T(M_{\varphi}, t) = 0$.
\end{theorem}

\vspace{0.2 cm}
\noindent
{\it Remark :} In Theorem 4.9 of \cite{BH}, D. Burghelea and S. Haller obtained the result corresponding to Theorem 3.4 only when $\varphi : M \rightarrow M$ is a diffeomorphism. However, we give an elementary proof by using the BFK-gluing formula when $\varphi$ is an isometry.

\vspace{0.5 cm}


\end{document}